\documentclass[10pt,a4paper]{article}
\usepackage{amsfonts}
\usepackage{amssymb,amsthm, amsmath,graphicx}
\usepackage{bbm}
\usepackage{graphicx}

\usepackage{xcolor}

\usepackage{url}
\usepackage{multirow}

\newtheorem{theorem}{Theorem}

\newtheorem{proposition}[theorem]{Proposition}
\newtheorem{corollary}[theorem]{Corollary}
\newtheorem{lemma}[theorem]{Lemma}

\theoremstyle{remark}

\def\CaP{\mathbf{P}}

\def\BfH{\mathcal{H}}

\def\FraC{\mathcal{C}}
\def\FraF{\mathcal{F}}

\def\FraS{\mathcal{S}}
\def\FraP{\mathcal{P}}
\def\FraT{\mathcal{T}}
\def\k{\mathbbmss{k}}
\def\N{\mathbb{N}}
\def\R{\mathbb{R}}

\def\Q{\mathbb{Q}}
\def\d{\mathrm{d}}

\def\int{\mathrm{int} }
\def\ap{\mathrm{Ap} }

\title{Computing  families of Cohen-Macaulay and Gorenstein rings}

\author{J. I. Garc\'{\i}a-Garc\'{\i}a\footnote{Departamento de Matem\'aticas, Universidad de C\'adiz,
E-11510 Puerto Real (C\'{a}diz, Spain). E-mail: ignacio.garcia@uca.es. Partially supported by MTM2010-15535 and Junta de Andaluc\'{\i}a group FQM-366. }\\
A. Vigneron-Tenorio\footnote{Departamento de Matem\'aticas, Universidad de C\'adiz,
E-11405 Jerez de la Frontera (C\'{a}diz, Spain). E-mail: alberto.vigneron@uca.es. Partially supported by MTM2007-64704, MTM2012-36917-C03-01
 and Junta de Andaluc\'{\i}a group FQM-366.}\\
}

\date{}

\begin{document}

\maketitle

\begin{abstract}
We characterize
Cohen-Macaulay and Gorenstein rings obtained from certain types of  convex body semigroups.
Algorithmic methods to check if a polygonal or circle semigroup is Cohen-Macaulay/Gorenstein are given. We also provide some families of Cohen-Macaulay and Gorenstein rings.


\smallskip
{\small \emph{MSC-class:} 20M14 (Primary),  20M05 (Secondary).}
\end{abstract}

\section*{Introduction}

Let $(S,+)$ be a finitely generated commutative monoid. If $\k$ is a field, we denote by $\k[S]$ the semigroup ring of $S$ over $\k$.
Note that $\k[S]$ is equal to $\bigoplus_{m \in S} \mathbbmss{k} \chi^m$ endowed with a multiplication which is $\mathbbmss{k}$-linear and such that $\chi^m \cdot \chi^n = \chi^{m+n}$, $m$ and $n \in S$ (see \cite{referencia_de_Pilar}).

Let $R$ be a Noetherian local ring, a finite $R$-module
$M\neq0$  is a Cohen-Macaulay module if ${\rm depth}(M)=\dim(M)$. If $R$ itself
is a Cohen-Macaulay module, then it is called a Cohen-Macaulay ring (see \cite{libro-C-M}).
In commutative algebra, a Gorenstein local ring is a Noetherian commutative local ring $R$ with finite injective dimension, as an $R$-module. There are many equivalent conditions most dealing with some sort of duality condition,
Bass in \cite{ubicuidad} publicized the concept of Gorenstein rings,
this concept is a special case of the more general concept of Cohen-Macaulay ring.
Both properties have been widely studied in ring theory, but if we try to search Cohen-Macaulay  and Gorenstein rings, few methods to obtain  such rings are found  (see \cite{huneke} and references therein).
One of the  contributions of this work is that it gives a method for constructing Cohen-Macaulay rings and it provides examples of Gorenstein rings.

To achieve our goal convex body semigroups are used (see \cite{ACBS}).
Let $F\subset \R^k_\ge$ be a convex body and $F_i=\{iX| X\in F\}$ with $i\in \N$, the convex body semigroup generated by $F$ is the semigroup $\FraF=\bigcup_{i=0}^{\infty} F_i\cap \N^k$;
 we study the cases when the semigroups are generated by convex polygons and circles.
According to a theorem of Hochster normal semigroup rings are Cohen-Macaulay (see \cite{libro-C-M} Theorem 6.3.5), but in our case all circle and polygonal semigroups are not normal
 which makes our study novel.
In this work  Cohen-Macaulayness and Gorensteiness  of the semigroup rings given by circles and polygons are characterized,  obtaining families fulfilling both properties.

The contents of this work are organized as follows.
Section \ref{sec1} contains some known results of Cohen-Macaulay and Gorenstein rings in terms of their associated semigroups.
In Section \ref{sec2} we give some results concerning the Cohen-Macaulay and Gorenstein properties for semigroups defined from convex bodies of $ \R^2 $.
In Sections \ref{sec3}, \ref{sec4} and \ref{sec5} the Cohen-Macaulayness and Gorensteiness  in affine circle semigroups and affine convex polygonal semigroups are characterized, giving in addition
 a  method for constructing Cohen-Macaulay semigroups and
examples of
Gorenstein semigroups.
Finally in Section \ref{sec6} we
discuss how effective is the use of  circle and polygonal semigroups in our problem and
present some conclusions about the use of them.

\section{Terminology and previous results}\label{sec1}

Let $G$ be a nonempty subset of $\R^k_\ge$ and denote by $L_{\Q_\geq}(G)$ the rational cone $\left\{\sum_{i=1}^p q_if_i| p\in\N, q_i\in \Q_{\geq}, f_i\in G \right\}$. Let  $Fr(L_{\Q_\geq}(G))$ be the boundary of $L_{\Q_\geq}(G)$ considering the usual topology of $\R^k$ and define the interior of $G$ as $G\setminus Fr(L_{\Q_\geq}(G))$, we denote it by $\int(G)$. Along this work it is used $\d(P)$ as the Euclidean distance between a point $P$ and the origin $O$.

Let $S\subset \N^k$ be the affine semigroup generated by $\{n_1,\ldots ,n_r,n_{r+1},\ldots ,n_{r+m}\}$.
A semigroup $S$ is called simplicial when $r=\dim (S)$ and $L_{\Q_\geq}(S)=L_{\Q_\geq}(\{n_1,\ldots ,n_r\})$. All semigroups appearing in this work are simplicial, so in the sequel we will assume such property.
For every $n\in S$ denote by $\ap(n)$ the set $\{s\in S| s-n\notin S\}$; this set is known as the Ap\'{e}ry set of $n$.
It is straightforward to prove that  $L_{\Q_\geq}(S)\cap \N^k$ is an affine semigroup, denote it by  $\FraC$.

The ring $\k[S]$ is isomorphic to the subring of $\k[t_1,\ldots ,t_{r+m}]$ generated by $t^s=t_1^{s_1}\cdots t_{r+m}^{s_{r+m}}$ with $s=(s_1,\ldots ,s_{m+r})\in S$ (see \cite{TrungHoa}).
A semigroup is called Cohen-Macaulay (C-M) if its associated semigroup ring $\k[S]$ is Cohen-Macaulay. In the same way $S$ is said to be Gorenstein if $\k[S]$ is also Gorenstein.

The following results show some characterizations of Cohen-Macaulay  rings  in terms of its associated semigroup.

\begin{theorem}\cite[Theorem 1.1]{RosalesCM}\label{C-M_Rosales}
The following conditions are equivalent:
\begin{enumerate}
\item $S$  is C-M.
\item For any $a,b\in S$ with $a+n_i=b+n_j$ ($1\le i\neq j\le r$), $a-n_j=b-n_i\in S.$
\end{enumerate}
\end{theorem}

When the subsemigroups are in $\N^2$ the above result can be reformulated.

\begin{corollary}\label{C-M}
Let $\FraS\subseteq \N^2$,
the following conditions are equivalent:
\begin{enumerate}
\item $S$ is C-M.
\item For all $a\in\FraC\setminus S$, $a+n_1$ or $a+n_2$ does not belong to $S$.
\end{enumerate}
\end{corollary}

\begin{proof}

Assume that there exists $a\in \FraC\setminus S$ such that $\alpha=a+n_2\in S$ and $\beta=a+n_1\in S$. Clearly $\alpha+n_1=\beta+n_2$, but $\alpha-n_2=\beta-n_1=a\not\in S$, which implies that $S$ is not C-M (see condition 2 in Theorem \ref{C-M_Rosales}).

Conversely, let now $\alpha,\beta$ be two elements of $S$ satisfying $\alpha+n_1=\beta+n_2$. Since $\alpha,\beta\in L_{\Q_\geq}(\{n_1,n_2\})$, they can be uniquely expressed as $\alpha=\lambda_1n_1+\lambda_2n_2$ and $\beta=\mu_1n_1+\mu_2n_2$ with $\lambda_1,\lambda_2,\mu_1,\mu_2\in \Q_{\geq}$.
So $\beta=\alpha+n_1-n_2=(\lambda_1+1)n_1+(\lambda_2-1)n_2$ obtaining that
$\lambda_2-1=\mu_2\geq 0$ and hence $\lambda_2\geq 1$. Analogously, we obtain that $\mu_1\geq 1$. Thus $\alpha-n_2=\beta-n_1\in L_{\Q_\geq}(\{n_1,n_2\})$. Taking now $a=\alpha-n_2=\beta-n_1$ and using the hypothesis,
since $a+n_1=\beta\in S$ and $a+n_2=\alpha\in S$, the element $a$ belongs to $S$. We conclude that $S$ is C-M.

\end{proof}

\begin{lemma}\label{lemma_no_C-M}
Let $S\subset \N^2$ be a simplicial affine semigroup such that $\int (\FraC) \setminus \int(S)$ is nonempty finite set, then $S$ is not C-M.
\end{lemma}

\begin{proof}

Let $a\in \int(\FraC)\setminus \int(S)$ such that $\d(a)=\max \{\d(a')| a'\in \int (\FraC) \setminus \int(S)\}$. The elements $a+n_1$ and $a+n_2$ are in $S$ and by Corollary \ref{C-M} the semigroup $S$ is not C-M.
\end{proof}

In \cite[Section 4]{RosalesCM} it is  given the following result that characterizes  Gorenstein semigroups.

\begin{theorem}\label{Gorenstein_rosales}
The following conditions are equivalent:
\begin{enumerate}
\item $S$ is Gorenstein.
\item $S$ is C-M and $\cap_{i=1}^r \ap(n_i)$ has a unique maximal element (with respect to the order defined by $S$).
\end{enumerate}
\end{theorem}

\section{Some results for convex body semigroups}\label{sec2}

Let $F\subset \R^r_\ge$ be a convex body and assume that $\FraF=\bigcup_{i=0}^{\infty} F_i\cap \N^r$ is an affine convex body semigroup minimally generated by $\{n_1,\ldots ,n_r,n_{r+1},\ldots ,n_{r+m}\}$ where $F_i=i\cdot F$ with $i\in \N$.
It can easily obtained an algorithm to check if an element $P$ belongs to $\FraF$. Just consider the ray $\tau$ defined by $P$, the set $\tau\cap F$ is a segment $\overline{AB}$ with $\d(A)\leq \d(B)$. To check if $P\in\FraF$, we consider $\{k\in\N|\frac{\d(P)}{\d(B)}\leq k \leq \frac{\d(P)}{\d(A)}\}$, if this set is nonempty then $P\in\FraF$.
The affine semigroup $L_{\Q_\geq}(\FraF)\cap \N^r$ is denoted by $\FraC$ and for $r=2$ it is a simplicial semigroup if and only if $F\cup\{O\}\subset \R^2_\ge$ has at least three non-aligned points. Assume $\FraF$ is simplicial, then denote by $\tau_1$ and $\tau_2$ its extremal rays with $\tau_1$ the ray with greater slope (note that $Fr(L_{\Q_\geq}(G))=\{\tau_1,\tau_2\}$) and by
$n_1\in \tau_1$ the element of $\FraF\cap \tau_1$ with less module, similarly define $n_2\in\tau_2$.

\begin{lemma}\label{lema_igual_interior_C-M}
Let  $\FraF\subset \N^2$ be a simplicial affine convex body semigroup not equal to $\FraC$ such that $\int (\FraC) = \int(\FraF)$. The semigroup $\FraF$ is C-M if and only if $\FraF\cap \tau_i=\langle n_i \rangle$ for $i=1,2$.
\end{lemma}

\begin{proof}

Assume that $\FraF$ is C-M and  $\FraF\cap \tau_1$ is not generated by $n_1$. In such case, $F\cap \tau_1$ is a segment and $(\FraC\setminus\FraF)\cap \tau_1$ has a finite number of elements. Take $P\in (\FraC\setminus \FraF)\cap \tau_1$ the element with largest module, then $P+n_1\in\FraF$. Now, since $int(\FraC)=int(\FraF)$, $P+n_2\in\FraF$. Thus $\FraF$ is not C-M, which is a contradiction.

Suppose now that $\FraF\cap \tau_i=\langle n_i \rangle$ for $i=1,2$.
Given $P\in\FraC\setminus \FraF$, we are going to prove that $P+n_1\not\in\FraF$ or $P+n_2\not\in\FraF$.
Assume without loss of generality that $P\in(\FraC\setminus\FraF)\cap \tau_1$, using that $\FraF\cap \tau_1=\langle n_1\rangle$ we obtain that $P+n_1\not \in\FraF$.

\end{proof}

Define the convex hull of a subset $G$ of $\R^n$ as the the intersection of all convex sets containing $G$, denote it by ${\rm ConvexHull}(G)$ and denote by $\overline{AB}$ the segment defined by  two points $A$ and $B$.
Let  $\BfH\subset \N^2$  be the set
$\left({\rm ConvexHull}(\{O,n_1,n_2,n_1+n_2\})\setminus \{ \overline{On_1},\overline{On_2},n_1+n_2\}\right)\cup \{O\}$.

\begin{lemma}\label{lema_apery}
Let  $\FraF\subset \N^2$ be a simplicial affine convex body semigroup such that $\int (\FraC) = \int(\FraF)$ and $\FraF\cap \tau_i=\langle n_i \rangle$ for $i=1,2$. Then $\ap(n_1)\cap \ap(n_2)=\FraF \cap \BfH.$
\end{lemma}

\begin{proof}
Trivially, $\FraF \cap \BfH\subset \ap(n_1)\cap \ap(n_2).$

Let $P\in \FraF\setminus \BfH$. If $P\in \int (\FraF)$, then $P-n_1\in \int(\FraF)$ or $P-n_2\in \int(\FraF)$, so $P\not \in  \ap(n_1)\cap \ap(n_2)$. If $P\in \FraF\cap \tau_i=\langle n_i \rangle$, then  $P-n_i\in \FraF$ and again $P\not \in  \ap(n_1)\cap \ap(n_2)$ .

\end{proof}

\begin{corollary}\label{corollary_igual_interior_Gorenstein}
Let  $\FraF\subset \N^2$ be a C-M simplicial affine convex body semigroup such that $\int (\FraC) = \int(\FraF).$ The semigroup $\FraF$ is Gorenstein if and only if there exists a unique maximal element in the set $\FraF \cap \BfH.$
\end{corollary}

\begin{proof}
By Theorem \ref{Gorenstein_rosales}, $\FraF$ is Gorenstein if and only if $\ap(n_1)\cap \ap(n_2)$ has a unique maximal element. By Lemma \ref{lema_apery}, we obtain that $\ap(n_1)\cap \ap(n_2)=\FraF \cap \BfH$ which proves the result.
\end{proof}

\section{C-M and Gorenstein affine circle semigroups}\label{sec3}

Let $C$ be the circle with center  $(a,b)$ and radius $r>0$ and as in \cite{ACBS}, define  $C_i$  as the circle with center $(ia,ib)$ and radius $ir$ and $\FraS=\bigcup_{i=0}^\infty C_i\cap \N^2$ the semigroup generated by $C$. Using the notation of the preceding sections, denote by $\FraC$ the positive integer cone $L_{\Q_\geq}(C\cap \R ^2_{\geq})\cap \N^2$, $\tau_1$ and $\tau_2$ the extremal rays of $\FraC$ and $\int(\FraC)=\FraC\setminus\{\tau_1,\tau_2\}.$
In \cite[Theorem 18]{ACBS} affine circle semigroups are characterized and it can be easily prove that
when $C\cap \R^2_\geq$ has at least two points the circle semigroup is simplicial.
In this section, we consider that $\FraS$ is a simplicial affine circle semigroup.
The following result characterizes C-M affine circle semigroups.

\begin{proposition}\label{circulos_C-M}
Assume that $\FraS\neq\FraC$, then
$\FraS$ is a C-M semigroup if and only if $\int (\FraS)=\int (\FraC)$ and  $\FraS\cap \tau_i=\langle n_i \rangle$ for $i=1,2$.
\end{proposition}

\begin{proof}
Let us suppose that $\FraS$ is C-M. Since $\FraS$ is a circle semigroup, if $\int(\FraC)\neq \int (\FraS)$, then  $\int(\FraC)\setminus \int (\FraS)$ is a nonempty finite set (see \cite[Lemma 17]{ACBS}). If this occurs, by Lemma \ref{lemma_no_C-M} $\FraS$ is not C-M and therefore $\int(\FraC)$ must be equal to $\int(\FraS)$.
Now, by Lemma \ref{lema_igual_interior_C-M} we obtain that $\FraS\cap \tau_i=\langle n_i \rangle$.

The other implication is obtained by Lemma \ref{lema_igual_interior_C-M}.
\end{proof}

The semigroup represented in Figure \ref{ejemplo_circulos_CM} is an instance of C-M circle semigroup.

\begin{figure}[h]
    \begin{center}
\begin{tabular}{|c|}\hline
\includegraphics[scale=0.35]{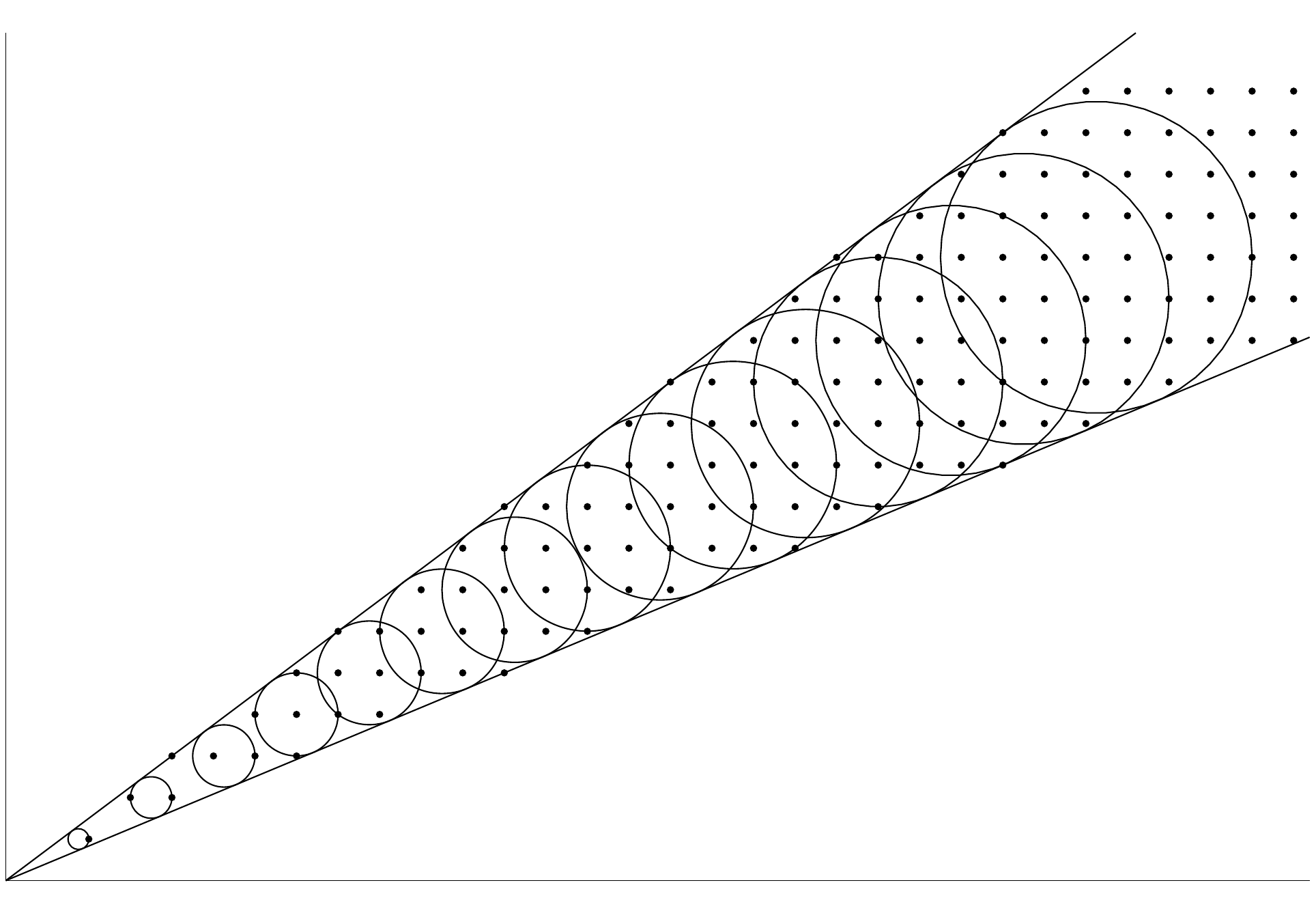}\\
Circle  center at $(7/4,1)$ and radius 1/4
\\ \hline
\end{tabular}
\caption{C-M circle semigroup.}\label{ejemplo_circulos_CM}
\end{center}
\end{figure}

The above characterization of C-M circle semigroups allows to obtain a characterization of Gorenstein affine circle semigroups.

\begin{corollary}\label{circulos_gorenstein}
Let $\FraS$ be a C-M circle semigroup.
$\FraS$ is Gorenstein if and only if there exists a unique maximal element in $\FraC \cap \BfH.$
\end{corollary}

\begin{proof}
Trivial by Corollary \ref{corollary_igual_interior_Gorenstein} and Proposition \ref{circulos_C-M}.
\end{proof}

\section{C-M affine convex polygonal semigroups}\label{sec4}

In this section we denote by $F\subset \R^2_{\geq}$ a compact convex polygon with vertices $\CaP=\{P_1,\ldots ,P_t\}$ arranged in counterclockwise direction and by $\FraP=\bigcup_{i=0}^{\infty} F_i\cap \N^2$ its associated semigroup. Affine convex polygonal semigroups are characterized in \cite[Theorem 14]{ACBS}.
From now on, we consider that $\FraP$ is a simplicial affine convex polygonal semigroup no equal to $\FraC$ and that the polygon $F$ is not a single segment. Once again, let $\tau_1$ and $\tau_2$ be the extremal rays of $\FraC$  assuming $\tau_1$ with a slope greater than $\tau_2$.

We define now the different sets we need to characterize and check the properties we are interested on, these will divide the cone of the semigroup $\FraP$ in several parts with different properties.
We distinguish two cases, $F\cap \tau_i$ is a point or it is formed by more that one  point.

Assume $F\cap \tau_1=\{P_1\}$ with $P_1\in \CaP$, define $j$ as the least positive integer such that $j\overline{P_{1}P_{t}} \cap (j+1)\overline{P_{1}P_{2}}$ is not empty and formed by a point, we call it $V_1$ (the existence of $j$ is proven in \cite{ACBS}
and can be easily computed).
Denote by $T_1$ the triangle with vertices $\{ O, P_1, V_1-jP_1 \},$ and by $\stackrel{\circ}{T_1}$ its topological interior. By \cite[Lemma 11]{ACBS}, for every $h\in\N$ with $h\geq j$ the points $h\overline{P_{1}P_{t}} \cap (h+1)\overline{P_{1}P_{2}}$ are in the same straight line which we denote by $\nu_1$. With this notation it is easy to prove that for every
$h\in \N,$ $\big((\stackrel{\circ}{T_1}\cup (\overline{OP_1}\setminus\{O,P_1\}))+hP_1\big)\cap \FraP=\emptyset$ and that if $h\geq j$ then
$\big((\stackrel{\circ}{T_1}\cup (\overline{OP_1}\setminus\{O,P_1\}))+hP_1+n_1\big)\cap \N^2=\Big(\big((\stackrel{\circ}{T_1}\cup (\overline{OP_1}\setminus\{O,P_1\}))+hP_1\big)\cap \N^2\Big)+n_1$ (note that $n_1$ is a multiple of $P_1$
and denote as $j_1$ the integer  verifying $j_1P_1=jP_1+n_1$).
 This implies that the elements of
$\mathcal{B}_1=\{ D+\lambda n_1 | D\in \overline{(j_1P_1)(V_1+n_1)} \text{ and } \lambda\in \Q_{\ge} \}\cap\N^2$ belong to $\FraP$ or they are in $\stackrel{\circ}{T_1}\cup (\overline{OP_1}\setminus\{O,P_1\})+\mu P_1$ with $\mu$ a nonnegative integer, in particular if $P\in \mathcal{B}_1\setminus \FraP$ then $P+n_1\not\in \FraP$ and if $P\in\mathcal{B}_1\cap \FraP$ then $P-n_1\in \FraP$. Denote $\Upsilon_1$ the finite set ${\rm ConvexHull} (\{O,j_1P_1,V_1+n_1, \nu_1\cap \tau_2 \}) \setminus \bigcup _{h < j_1,h\in\N} \big((\stackrel{\circ}{T_1}\cup (\overline{OP_1}\setminus\{O,P_1\}))+hP_1\big).$

Analogously, if $F\cap \tau_2$ is a point for instance $P_1\in \CaP$ then there exists the least integer $j$ such that $j\overline{P_{1}P_{2}} \cap (j+1)\overline{P_{1}P_{t}}$ is equal to a single point $V_2$ in this case. Denote by $T_2$ the triangle with vertices $\{ O, P_1, V_2-jP_1 \},$ by $\nu_2$ the line containing the points $\{h\overline{P_{1}P_{2}} \cap (h+1)\overline{P_{1}P_{t}}| h\ge j, h\in\N\}$, by $j_1$ the integer fulfilling that $j_1P_1=jP_1+n_2$ and by $\mathcal{B}_2$ to the set $\{ D+\lambda n_2 | D\in \overline{(j_1P_1)(V_2+n_2)} \text{ and } \lambda\in \Q_{\ge} \}\cap \N^2$. All these set verify analogous properties to the properties verified by the sets defined in the above paragraph. Denote in this case by $\Upsilon_2$ to the finite ${\rm ConvexHull} (\{O,j_1P_1,V_2+n_2, \nu_2\cap \tau_1 \}) \setminus \bigcup _{h < j_1,h\in\N} \big((\stackrel{\circ}{T_2}\cup (\overline{OP_1}\setminus\{O,P_1\}))+hP_1\big).$

If $F\cap \tau_i$ is a segment for some  $i$, we take  $\nu_i=\tau_i$ and  $\Upsilon_i=\{O\}.$
With the above construction we define the finite set  $\Upsilon '=\{ P\in \Upsilon _1\cap \N^2| P+n_1,P+n_2\in \FraP \}$,  $\Upsilon ''=\{ P\in \Upsilon _2\cap \N^2| P+n_1,P+n_2\in \FraP \}$ and   $\Upsilon = (Q+\L_{\Q_\geq}(F))\cap \N^2$ where $Q\in \L_{\Q_\geq}(F)$ is the point $\nu_1\cap \nu_2$.
Note that the boundary of the set $\Upsilon$ intersects with two different sides of the polygon  $i_0F$ when $i_0\gg 0$ and therefore the set  $(\FraC\cap \Upsilon)\setminus \FraP$ is finite.
This construction is illustrated in Figure \ref{franja}.
\begin{figure}[h]
    \begin{center}
\begin{tabular}{|c|}\hline
\includegraphics[scale=0.45]{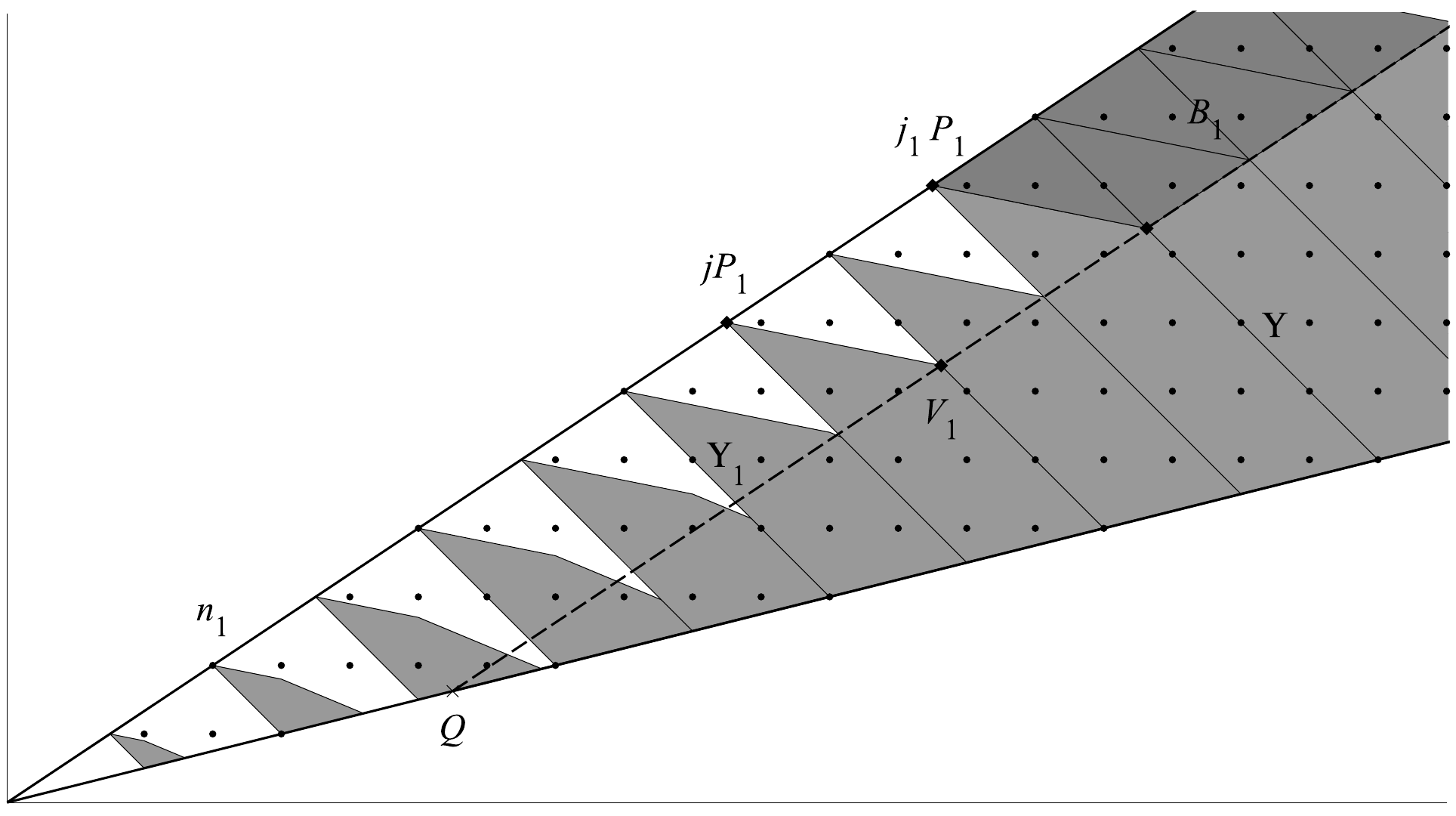} \\
Polygon $=\{( 2, 0.9),( 2, 0.5),( 2.6, 0.65),( 1.5, 1)\}$ \\ \hline
\end{tabular}
\caption{Construction for $F\cap \tau_1=\{P_1\}$ and $F\cap \tau_2$ a segment.}\label{franja}
\end{center}
\end{figure}
It can be observed that $\FraC$ is the union of the integer points of the set $\Upsilon$ and the stripes bounded by the parallel lines $\tau_1$  and $\nu_1$, and $\tau_2$ and $\nu_2$. In these partitions the following properties are satisfied:
\begin{enumerate}
\item
    Every point $P\in\FraC$ belonging to the stripes defined by the parallel lines $\tau_1$ and $\nu_1$, and $\tau_2$ and $\nu_2$ (with $F\cap\tau_i$ a single point) verifies:
    \begin{enumerate}
    \item If $P\in (\mathcal{B}_1\cup \mathcal{B}_2)\setminus\FraP$ then $P+n_1$ or $P+n_2$ does not belong to  $\FraP.$
    \item If $P\in \mathcal{B}_i\cap\FraP$ for some  $i$ then  $P-n_i$ belongs to  $\FraP.$
    \item
If $P\notin \mathcal{B}_1\cup \mathcal{B}_2$, then $P$ belongs to $\Upsilon_1 \cup \Upsilon _2$
or $(\stackrel{\circ}{T_i}\cup ((T_i\cap \tau_i)\setminus\{O,P_i\}))+\mu P_i $ with $\mu$ a nonnegative integer and $\{P_i\}=F\cap \tau_i$.
In case $P\in \Upsilon_1\cup\Upsilon_2$ but $P\not \in\FraP$  by Corollary \ref{C-M},  $\FraP$ is not C-M.
If $P$ belongs to
$(\stackrel{\circ}{T_i}\cup ((T_i\cap \tau_i)\setminus\{O,P_i\}))+\mu P_i $ we obtain that  $P+n_1$ or $P+n_2$ are not in $\FraP$ ($\FraP$ is still capable of being C-M).
\end{enumerate}
\item
    For $\FraC\cap \Upsilon$, if $\Upsilon \setminus\FraP$ is not empty then $\FraP$ is not C-M. It is enough to take $P$ as the element of $\Upsilon\setminus\FraP$ of greatest distance to the origin in order to verify that $P+n_1,P+n_2\in \FraP$ obtaining by Corollary \ref{C-M} that $\FraP$ is not C-M.
\end{enumerate}

\begin{lemma}\label{lemaUpsilon}
In the above conditions, for all point $P\in \FraC\setminus \FraP$ such that $P\notin  \Upsilon\cup \Upsilon'\cup \Upsilon'',$ $P+n_1$ or $P+n_2$ is not in $\FraP.$
\end{lemma}

\begin{proof}
Immediate from the above construction.
\end{proof}

If $F\cap \tau_1$ and $F\cap \tau_2$ are both segments, by \cite[Proposition 13]{ACBS} the set $\int (\FraC) \setminus \int(\FraP)$ is finite and by Lemma \ref{lemma_no_C-M} if it is nonempty $\FraP$ is not C-M.
If not, we have $\int (\FraC) = \int(\FraP)$,  $\FraC\cap \tau_{i}\neq \FraP\cap \tau_{i}$ for some $i=1,2$ and when this occurs  $(\FraC \setminus \FraP) \cap \tau_{i}$ is finite. Taking now $P$ the element of $(\FraC\setminus \FraP)\cap \tau_{i}$ of greatest distance to the origin we have that $P+n_1$ and $P+n_2$ are both in $\FraP$ and therefore $\FraP$ is not C-M.
Henceforth we will assume that $F\cap \tau_1$ and $F\cap \tau_2$ are not simultaneously segments.

\begin{theorem}\label{corolario_polygon1}
Let $\FraP$ be a simplicial affine convex polygonal semigroup such that $F\cap \tau_1$ and $F\cap \tau_2$ are not both segments. Then
\begin{enumerate}
\item \label{caso1} if $\int (\FraC) = \int(\FraP),$ the semigroup $\FraP$ is C-M if and only if $\FraP\cap \tau_i=\langle n_i \rangle$ for $i=1,2$,
\item \label{caso2} if $\int (\FraC) \neq \int(\FraP),$ the semigroup $\FraP$ is C-M if and only if
$\Upsilon\cup \Upsilon'\cup \Upsilon''\subset \FraP.$
\end{enumerate}
\end{theorem}

\begin{proof}

Without loss of generality we can suppose that $F\cap \tau_1$ is a single point.

Case \ref{caso1} is obtained by Lemma \ref{lema_igual_interior_C-M}.

Consider now that $\int (\FraC) \neq \int(\FraP)$ and the set $\Upsilon\cup \Upsilon'\cup \Upsilon''$ is  included in $\FraP$. For all $P$ belongs to $\FraC\setminus \FraP,$ $P\notin \Upsilon\cup \Upsilon'\cup \Upsilon''$ and then $P+n_1$ or $P+n_2$ is not in $\FraP$ (Lemma \ref{lemaUpsilon}). So $\FraP$ is C-M (Corollary \ref{C-M}).

Assume now that $\FraP$ is C-M, $\int (\FraC) \neq \int(\FraP)$ and $(\Upsilon\cup \Upsilon'\cup \Upsilon'')\setminus \FraP\neq \emptyset.$
If $\Upsilon \setminus \FraP$ is not empty, then $\FraP$ is not C-M. Analogously, if there exists
$P\in (\Upsilon'\cup \Upsilon'')\setminus \FraP,$ we have that $P+n_1,P+n_2\in \FraP,$ and $\FraP$ is not C-M. Thus, $\Upsilon\cup \Upsilon'\cup \Upsilon'' \subset \FraP.$
\end{proof}

The following result provides a method to obtain Cohen-Macaulay rings.
Note that if $F$ is a triangle, the sets $\Upsilon'$ and $\Upsilon''$ are included in $\FraP$,
thereby we prove now that $\Upsilon\subset \FraP.$

\begin{corollary}\label{triangulos_C-M}
Any affine polygonal semigroup generated by a triangle with rational vertices is C-M.
\end{corollary}

\begin{proof}

Consider the triangle $T$ with vertices $\{A,C,B\}\subset \Q_{\geq}^2$ arranged in counterclockwise,  such that $T\cap \tau_1=A$, $L_{\Q_\geq}(T)=L_{\Q_\geq}(\{A,C\})$ and let $\FraT$ the affine semigroup generated by $T$.

If $C$ is not in the ray defined by the origin and $B$, consider
$D$ the point obtained as the intersection of the segment $\overline{AC}$ and the ray defined by the origin and $B$, define $T'$ the triangle with vertices $\{A,B,D\}$ and $T''$ the triangle with vertices $\{B,C,D\}$. Note that $\Upsilon$ is the union of the corresponding $\Upsilon$ sets of the semigroups generated by $T'$ and $T''$.
This situation if illustrated in Figure \ref{triangulo_C-M_particion}.
\begin{figure}[h]
    \begin{center}
\begin{tabular}{|c|}\hline
\includegraphics[scale=0.4]{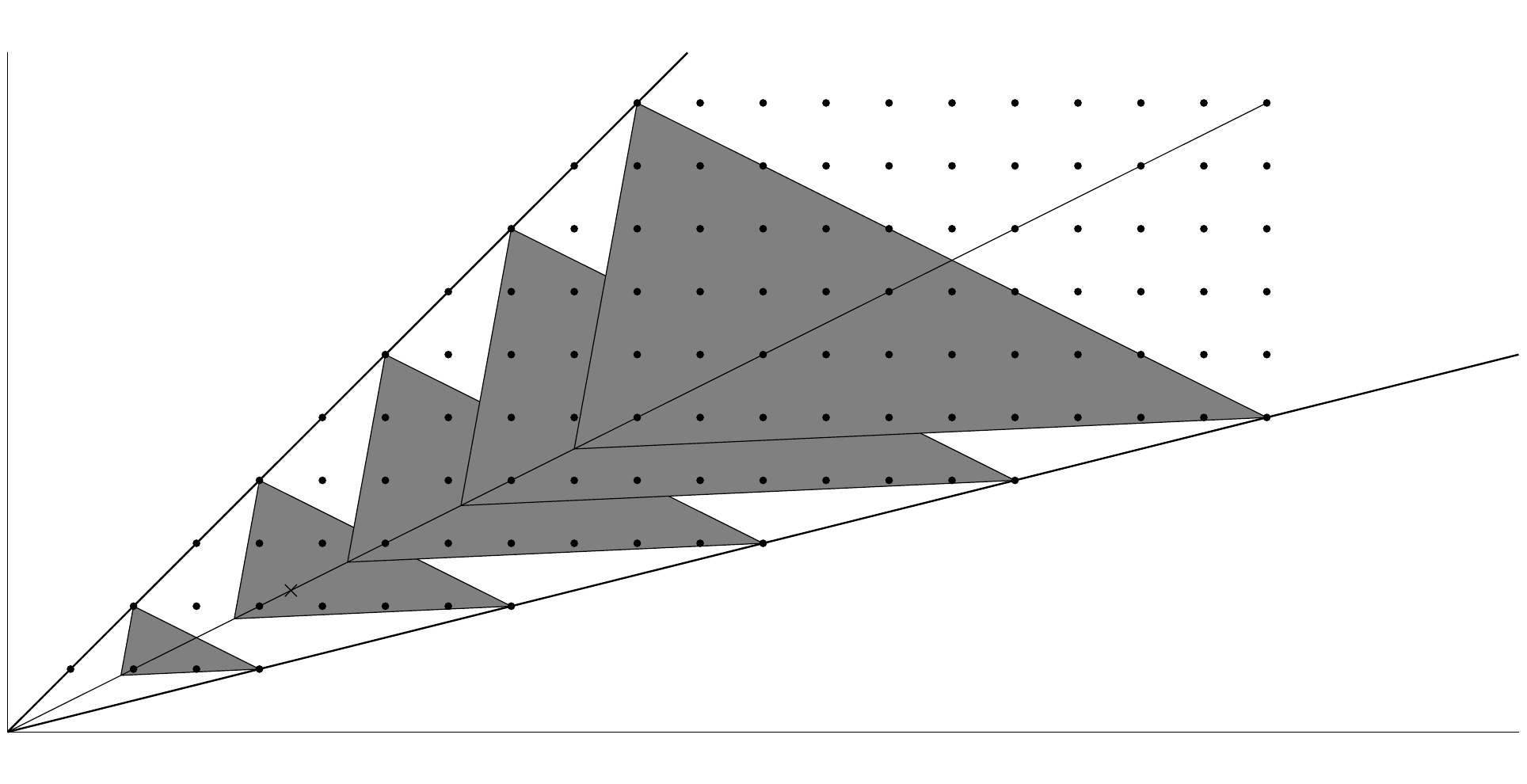} \\ \hline
\end{tabular}
\caption{Initial partition.}\label{triangulo_C-M_particion}
\end{center}
\end{figure}

We can simplify our situation assuming that  $O$, $B$ and $C$ are aligned and $\d(B)<\d(C)$. Let now $j$ be the smallest nonnegative integer fulfilling that $jT\cap (j+1)T\neq \emptyset$ and denote by $T'''$ the triangle $\{jA,(j+1)A,\overline{(jA)(jC)}\cap \overline{((j+1)A)((j+1)B)}-j\overrightarrow{OA}.$
So $\int (\FraT)= \int(\FraC)\setminus \{\cup _{i\ge 0} \stackrel{\circ}{iT'''} \},$ and the set $\Upsilon$ is determined by the straight line containing $OC$ and the straight line defined by the vertices of the triangles $iT'''$ which are not in the rays of $\FraC$, furthermore the intersection of this two straight lines is an element of a triangle $jT$ and this lead us to say that $\Upsilon \subset \FraP$. Figure \ref{triangulo_C-M_mitad} illustrates this setting.
\begin{figure}[h]
    \begin{center}
\begin{tabular}{|c|}\hline
\includegraphics[scale=0.35]{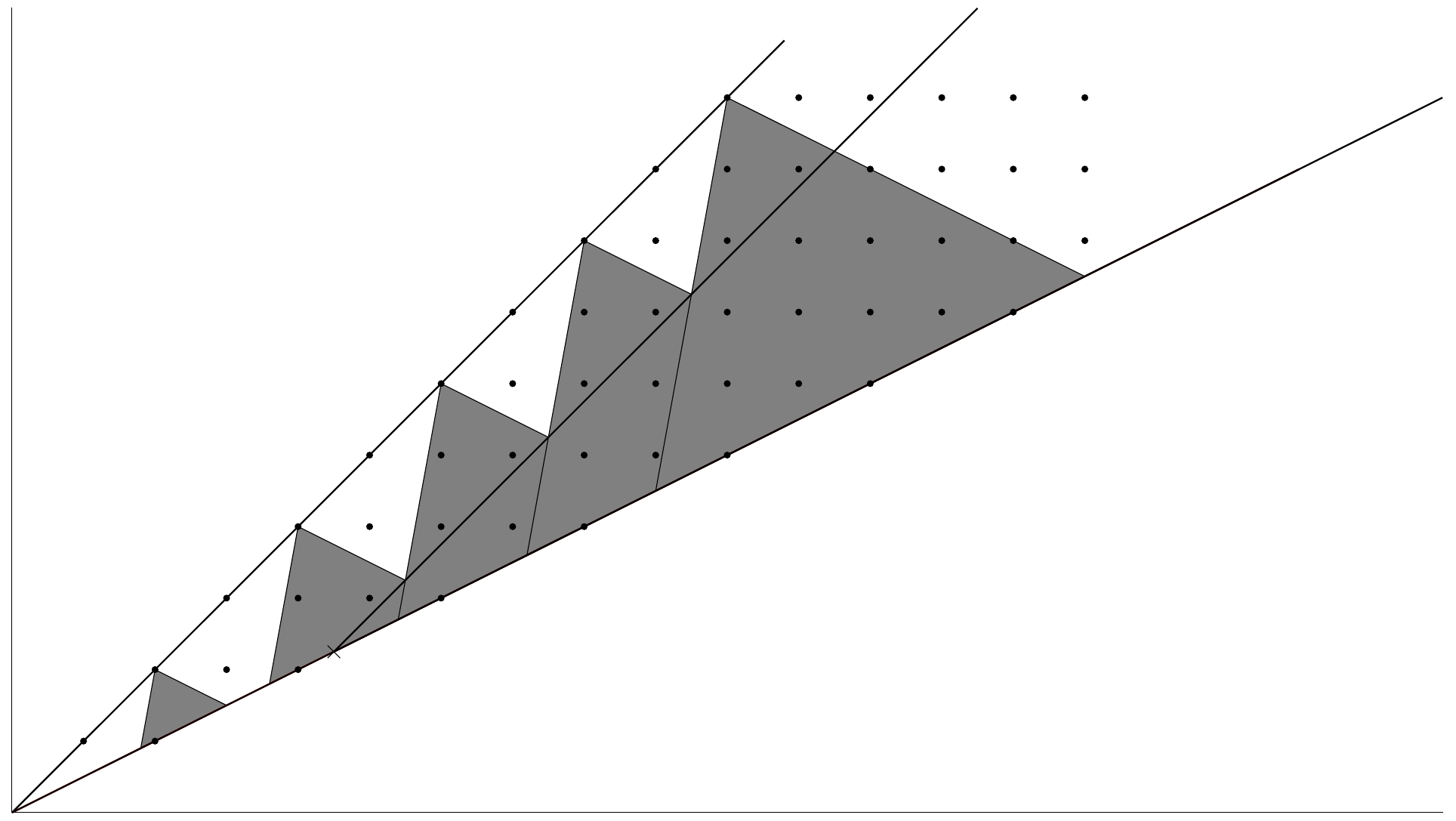} \\ \hline
\end{tabular}
\caption{Set $\Upsilon$.}\label{triangulo_C-M_mitad}
\end{center}
\end{figure}
Thus if $\int (\FraC)\neq \int (\FraT)$ by Theorem \ref{corolario_polygon1} case \ref{caso2} the semigroup $\FraT$ is C-M.

If $\int (\FraC)= \int (\FraT)$ and $\FraT=\FraC$, then by Corollary \ref{C-M} the semigroup $\FraT$ is C-M.
Assume now that $\int (\FraC)=\int (\FraT)$, $\FraT\neq\FraC$  and that $\FraT\cap \tau_2$ is generated by at least two elements.
In particular, if $a\in \FraC$ with $\FraC\cap \tau_2=\langle a\rangle$, then $a\in \overline{OQ}\setminus\FraP$ and therefore $a+n_1\notin \FraP$ which contradicts the fact that $\int (\FraC)=\int (\FraT)$.
So, in any case $\FraT\cap \tau_1$ and $\FraT\cap \tau_2$ are both generated by a unique element and by Theorem \ref{corolario_polygon1} case \ref{caso1} the semigroup $\FraT$ is C-M.
\end{proof}

\section{Gorenstein affine convex polygonal semigroups}\label{sec5}

Once C-M polygonal semigroups are characterized, we characterize Gorenstein polygonal semigroups. Since this characterization depends on the set $\ap(n_1)\cap \ap(n_2)$ (Theorem \ref{Gorenstein_rosales}), we start by describing this set.
 Define $\BfH ',$ $\BfH ''$ and $\BfH '''$ the finite sets $\Upsilon_1 \cap \{P\in \N^2| P-n_1,P-n_2\notin \FraP\},$ $\Upsilon_2 \cap \{P\in \N^2| P-n_1,P-n_2\notin \FraP\}$ and ${\rm ConvexHull}(\{Q,Q+n_1,Q+n_2,Q+n_1+n_2\}) \cap \{P\in \N^2| P-n_1,P-n_2\notin \FraP\}$, respectively.

\begin{lemma}\label{lema_apery_poligonal}
Let  $\FraP$ be a C-M simplicial affine polygonal semigroup then $\ap(n_1)\cap \ap(n_2)=\FraP \cap (\BfH'\cup \BfH''\cup \BfH''').$
\end{lemma}

\begin{proof}
It is straightforward that $\FraP \cap (\BfH'\cup \BfH''\cup \BfH''')\subset \ap(n_1)\cap \ap(n_2).$

Let $P\in\FraP\setminus (\BfH'\cup \BfH''\cup \BfH''')$, we see that $P\not\in \ap(n_1)\cap \ap(n_2)$.
Since $P\in \mathcal{B}_1\cup \mathcal{B}_2\cup (\Upsilon\setminus\BfH''')$, we obtain $P-n_1$ or $P-n_2$ belongs to $\FraP$ and therefore
$P\not\in \ap(n_1)\cap \ap(n_2)$.
Thus $\ap(n_1)\cap \ap(n_2)\subseteq \FraP\cap (\BfH'\cup \BfH''\cup \BfH''').$

\end{proof}

The following result characterized Gorenstein simplicial affine convex polygonal semigroups.

\begin{theorem}\label{polygonal_Gorenstein}
Under the assumption that $\FraP$ is C-M, the semigroup $\FraP$ is Gorenstein if and only if there exists a unique maximal element in the finite set $\FraP\cap (\BfH'\cup \BfH''\cup \BfH''').$
\end{theorem}

\begin{proof}
By Lemma \ref{lema_apery_poligonal} and Theorem \ref{Gorenstein_rosales}.
\end{proof}

Using the explicit description of the set $\ap(n_1)\cap \ap(n_2)$  given in Lemma \ref{lema_apery_poligonal}, it can be obtained an algorithm to check if a C-M polygonal semigroup is Gorenstein:
first we must compute $n_1$ and $n_2$ which can be done directly from the vertices of $F$; the elements $V_1$, $V_2$ and $Q$ are also required and can be computed using the results of Section \ref{sec4}; finally we must compute $\Upsilon_1$, $\Upsilon_2$ and $\Upsilon$ and from these sets $\BfH'$, $\BfH''$ and $\BfH'''$.
Note that the required elements can be computed by using elementary geometry and algebra. Besides, to determine the sets $\BfH'$, $\BfH''$ and $\BfH'''$ it is only required an algorithm to check the belonging of an element to the semigroup $\FraP$. That algorithm was introduced in Section \ref{sec2}, it uses
 the equations of the sides of $F$ and   its time
complexity is linear on the number of vertices when  they are arranged in counterclockwise.

From Corollary \ref{triangulos_C-M} and Theorem \ref{polygonal_Gorenstein} it is possible to construct nonnormal Gorenstein semigroup rings.
For instance we can generate a family of Gorenstein rings by using the triangles with vertex set $$\{(4,0),(4+2k,0),(4+k,k)\},$$
with  $k$ an integer greater than or equal to $2$. In Table \ref{interseccion_apery} we determine explicitly the intersection  $\ap(n_1)\cap \ap(n_2).$
\begin{table}[h]
\begin{center}
\begin{tabular}{|c|c|c|}
\hline
$\ap(n_1)\cap \ap(n_2)\cap \{y=0\}$ & = & $\{
(0,0),(5,0),(6,0), (7,0)\}$ \\ \hline
$\ap(n_1)\cap \ap(n_2)\cap \{{y=1\}}$ & = & $\{(5,1) , (6,1) , (7,1) , (8,1)\}$ \\ \hline
\vdots & \vdots & \vdots \\ \hline
\multirow{2}{*}{$\ap(n_1)\cap \ap(n_2)\cap \{y=k-2\}$} & = & $\{(2+k,k-2) , (3+k,k-2) ,$\\
& & $(4+k,k-2) , (5+k,k-2) \}$\\ \hline
\multirow{2}{*}{$\ap(n_1)\cap \ap(n_2)\cap \{y=k-1\}$} & = & $\{(3+k,k-1) , (4+k,k-1) ,$\\
 & & $(5+k,k-1) , (10+k,k-1) \}$\\ \hline
$\ap(n_1)\cap \ap(n_2)\cap \{y\ge k\}$ & = & $\emptyset$\\ \hline
\end{tabular}
\caption{$\ap(n_1)\cap \ap(n_2)$.}\label{interseccion_apery}
\end{center}
\end{table}
In the above table we can see that the elements of  $\ap(n_1)\cap \ap(n_2)$
are of the form $(4+k-i+j,k-i)$ with $j=0,\ldots ,3,$ and  $i=2,\ldots ,k-1.$
Thus since
 $(10+k,k-1)=(5,0)+(5+k,k-1)=(6,0)+(4+k,k-1)=(7,0)+(3+k,k-1)$ and $(10+k,k-1)=(4+k-i+j,k-i)+\big(4+k-(k-i+1)+(3-j),i-1\big)$ for all  $i=2,\ldots ,k-1$ and  $j=0,\ldots,3$
we have that  $(10+k,k-1)$ is the unique maximal element of  $\ap(n_1)\cap \ap(n_2)$.
Therefore all these polygonal semigroups are  Gorenstein. In Figure \ref{gorenstein1} it is represented the polygonal semigroup obtained when $k=3$.
\begin{figure}[h]
    \begin{center}
\begin{tabular}{|c|}\hline
\includegraphics[scale=0.4]{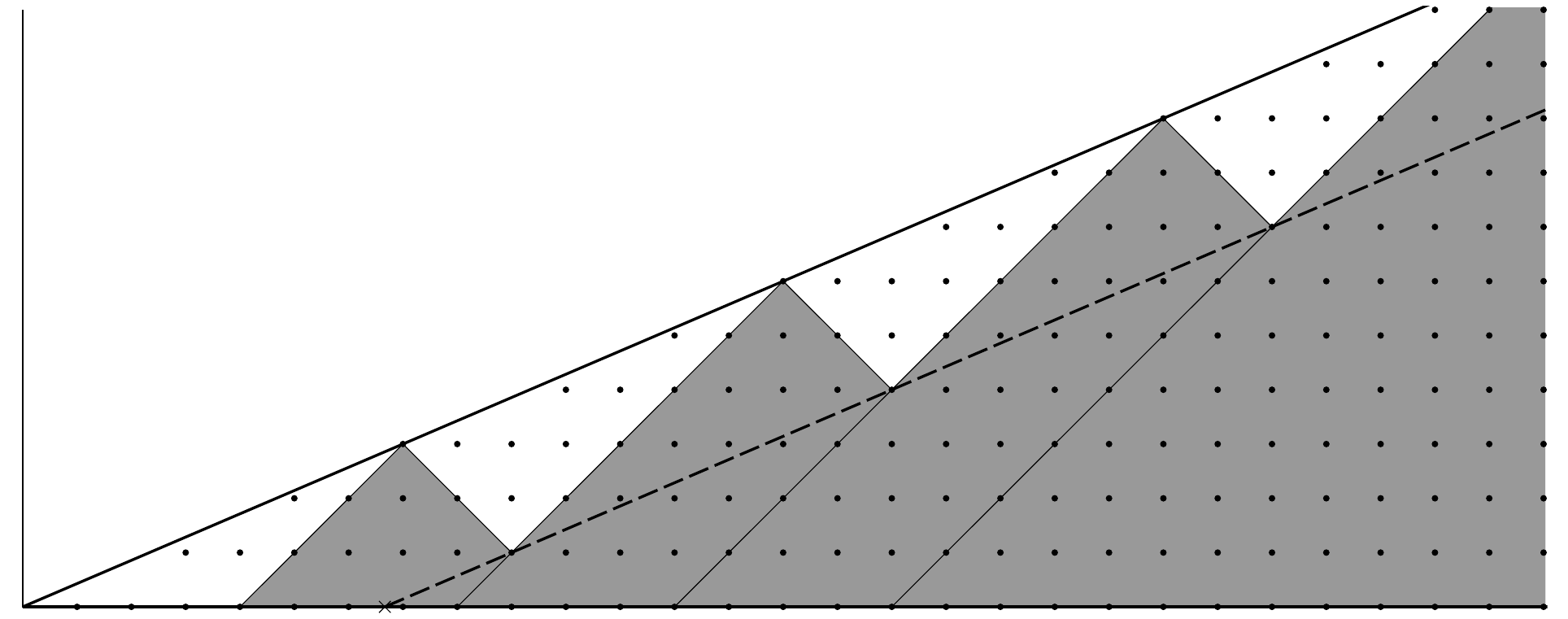} \\
Polygon$=\{(4, 0), (7, 3), (10, 0)\}$ \\
$\ap(n_1)\cap \ap(n_2)=\{O,(5,0),(6,0),(7,0),(5,1),(6,1),$\\
$(7,1),(8,1),(6,2),(7,2),(8,2),(13,2)\}$ \\ \hline
\end{tabular}
\caption{Example of Gorenstein polygonal semigroup.}\label{gorenstein1}
\end{center}
\end{figure}

\section{Some conclusions}\label{sec6}

Determine whether a circle or polygonal semigroup is C-M and/or Gorenstein is very simple because the elements needed to check the characterizations given in this paper (Proposition \ref{circulos_C-M}, Corollary \ref{circulos_gorenstein} and Theorems \ref{corolario_polygon1} and \ref{polygonal_Gorenstein}) can be obtained using basic geometry, all of this is mainly due to the fact that  the belonging of one element in one of these semigroups can be easily performed  from the polygon or circle that defines the semigroup. Thus convex body semigroups and all the theory developed to check their properties allow to generate instances of Cohen-Macaulay and Gorenstein semigroup rings, in particular using triangles many examples can be easily obtained.
Moreover, we think that convex body semigroups are a good first place to search for semigroups verifying certain properties not only the two properties characterized in this work.


\begin{thebibliography}{12}

\bibitem{ubicuidad}
\textsc{H. Bass}.
\newblock \emph{On the ubiquity of Gorenstein rings}.
\newblock Mathematische Zeitschrift (1963) 82, 8-–28.

\bibitem{referencia_de_Pilar}
\textsc{E. Briales, A. Campillo, C. Mariju\'{a}n, P. Pis\'{o}n}.
\newblock \emph{Minimal systems of generators for ideals of semigroups.}
\newblock J. Pure Appl. Algebra 124 (1998), no. 1-3, 7-–30.


\bibitem{libro-C-M}
\textsc{W. Bruns, J. Herzog}
\newblock Cohen-Macaulay rings.
\newblock  Cambridge Studies in Advanced Mathematics, 39, \emph{Cambridge University Press}, 1993.

\bibitem{ACBS}
\textsc{J.I. Garc\'{\i}a-Garc\'{\i}a, M.A. Moreno-Fr\'{\i}as, A. S\'{a}nchez-R.-Navarro, A. Vigneron-Tenorio}.
\newblock \emph{Affine convex body semigroups.}
\newblock To appear in Semigroup Forum. DOI 10.1007/s00233-012-9460-9.

\bibitem{huneke}
\textsc{C. Huneke}.
\newblock \emph{Hyman Bass and Ubiquity: Gorenstein Rings}.
\newblock  Contemp. Math., 55--78, 243, Amer. Math. Soc., 1999.

\bibitem{RosalesCM}
\textsc{J.C. Rosales, P.A. Garc\'{\i}a-S\'anchez}.
\newblock \emph{On Cohen-Macaulay and Gorenstein simplicial affine
semigroups.}
\newblock  Proc. Edinburgh Math. Soc. (2) 41 (1998), no. 3, 517--537.

\bibitem{Schrijver}
\textsc{A. Schrijver}.
\newblock Theory of linear and integer programming.
\newblock Wiley-Interscience Series in Discrete Mathematics. A Wiley-Interscience Publication. \emph{John Wiley \& Sons, Ltd., Chichester,} 1986.

\bibitem{TrungHoa}
\textsc{N.V. Trung, L.T. Hoa}.
\newblock \emph{Affine semigroups and Cohen-Macaulay rings generated by monomials.}
\newblock Trans. Amer. Math. Soc. 298 (1986), no. 1, 145--167.

\end{thebibliography}
\end{document}